\documentclass[reqno]{amsart}
\usepackage{amssymb}

\usepackage[colorlinks=true]{hyperref}
\hypersetup{urlcolor=blue, citecolor=red}

\newtheorem{Theorem}{Theorem}[section]

\newtheorem{Lemma}[Theorem]{Lemma}
\newtheorem{Proposition}[Theorem]{Proposition}
\newtheorem{Definition}[Theorem]{Definition}
\newtheorem{Remark}[Theorem]{Remark}

\def\C {\mathbb C}

\def\R {\mathbb R}

\newcommand{\<}{\langle}
\renewcommand{\>}{\rangle}
\renewcommand{\(}{\left(}
\renewcommand{\)}{\right)}

\newcommand{\p}{\partial}
\newcommand{\Vol}{\operatorname{Vol}}
\newcommand{\sgn}{\operatorname{sgn}}
\newcommand{\id}{\operatorname{Id}}

\begin{document}
\title[Range of the attenuated magnetic ray transform]{On the range of the attenuated magnetic ray transform for connections and Higgs fields}

\author[Gareth Ainsworth]{Gareth Ainsworth}
\address{Trinity College, Cambridge, CB2 1TQ, United Kingdom}
\email{gtrain137@gmail.com}

\author[Yernat M. Assylbekov]{Yernat M. Assylbekov}
\address{University of Washington, Department of Mathematics, Seattle, WA 98195-4350, USA}
\email{y\_assylbekov@yahoo.com}

\begin{abstract}
For a two-dimensional simple magnetic system, we study the attenuated magnetic ray transform $I_{A,\Phi}$, with attenuation given by a unitary connection $A$ and a skew-Hermitian Higgs field $\Phi$. We give a description for the range of $I_{A,\Phi}$ acting on $\C^n$-valued tensor fields.
\end{abstract}

\maketitle

\section{Introduction}
\subsection{Magnetic flows}
Consider a compact oriented Riemannian manifold $(M,g)$ with boundary. Let
$\pi:TM\to M$ denote the canonical projection,
$\pi:(x,v)\mapsto x$,
$x\in M$, $v\in T_xM$.
Denote by $\omega_0$ the canonical symplectic form on $TM$, which
is obtained by pulling back the canonical symplectic form of $T^*M$
via the Riemannian metric.
Let $H:TM\to\mathbb R$ be defined by
$$
H(x,v)=\frac 12|v|^2_{g(x)},\quad (x,v)\in TM.
$$
The Hamiltonian flow of $H$ with respect to $\omega_0$
gives rise to the geodesic flow of $(M,g)$.
Let $\Omega$ be a closed $2$-form on $M$ and consider
the new symplectic form $\omega$ defined as
$$
\omega=\omega_0+\pi^*\Omega.
$$
The Hamiltonian flow of $H$ with respect to $\omega$ gives rise to
the {\it\bfseries magnetic geodesic flow}
$\varphi_t:TM\to TM$.
This flow models the motion of
a unit charge of unit mass in a magnetic field
whose {\it\bfseries Lorentz force} $Y:TM\to TM$
is the bundle map uniquely determined by
$$
\Omega_x(\xi,\eta)=\langle Y_x(\xi),\eta\rangle
$$
for all $x\in M$ and $\xi,\eta\in T_xM$. Every trajectory of the magnetic flow
is a curve on $M$ called a {\it\bfseries magnetic geodesic}. Magnetic geodesics satisfy Newton's law of motion
\begin{equation}\label{m-geodesic}
\nabla_{\dot\gamma}\dot\gamma=Y_\gamma(\dot\gamma).
\end{equation}
Here $\nabla$ is the Levi-Civita connection of $g$. The triple $(M,g,\Omega)$ is said to be a {\it\bfseries magnetic system}. In the absence of a magnetic field, that is $\Omega=0$, we recover the ordinary geodesic flow
and ordinary geodesics. Note that magnetic geodesics are not time reversible, unless $\Omega=0$. Note also that magnetic geodesics have constant speed, and hence are restricted to a specific energy level - we will consider the magnetic flow on the unit sphere bundle $SM=\{(x,v)\in TM:|v|=1\}$. This is not a restriction at all from a dynamical point of view, since other energy levels may be understood by simply changing $\Omega$ to $c \Omega$, where $c\in\R$.

Magnetic flows were first considered in \cite{AS, Ar} and it was shown in \cite{ArG, Ko, N1, N2, NS, PP} that they are related to dynamical systems, symplectic geometry, classical mechanics and mathematical mechanics. Inverse problems related to magnetic flows were studied in \cite{Ain,DPSU,DU}

Let $\Lambda$ stand for the second fundamental form of $\partial M$
and $\nu(x)$ for the inward unit normal to $\partial M$ at $x$.
We say that $\partial M$ is {\em\bfseries strictly magnetic convex}
if
\begin{equation}\label{strict-conv}
\Lambda(x,\xi)>\langle Y_x(\xi),\nu(x)\rangle
\end{equation}
for all $(x,\xi)\in S(\partial M)$.
Note that if we replace $\xi$ by $-\xi$, we can put an absolute
value in the right-hand side of \eqref{strict-conv}.
In particular, magnetic convexity is stronger
than the Riemannian analogue.

For $x\in M$, we define the {\em\bfseries magnetic
exponential map} at $x$ to be the partial map
$\exp^\mu_x:T_xM\to M$ given by
$$
\exp^\mu_x(t\xi)=\pi\circ\varphi_t(\xi),\quad t\ge 0,\ \xi\in S_xM.
$$
In \cite{DPSU} it is shown that, for every $x\in M$,
$\exp^\mu_x$  is a $C^1$-smooth partial map on $T_xM$
which is $C^\infty$-smooth on $T_xM\setminus\{0\}$.

\begin{Definition}{\rm
We say that a magnetic system $(M,g,\Omega)$ is a {\em\bfseries simple magnetic system} if $\partial M$ is strictly magnetic convex and the magnetic exponential map
$\exp^\mu_x:(\exp^\mu_x)^{-1}(M)\to M$
is a diffeomorphism for every $x\in M$.}
\end{Definition}

In this case $M$ is diffeomorphic to the unit ball of Euclidean space, and therefore $\Omega=d\alpha$ for some 1-form $\alpha$ on $M$. We will denote such a simple magnetic system by $(M,g,\alpha)$. This definition is a generalization of the notion of a simple Riemannian manifold. The latter naturally arose in the context of the boundary rigidity problem \cite{Mi}. Throughout this paper we only consider simple magnetic systems.

\subsection{Attenuated magnetic ray transform for a unitary connection and Higgs field}
We consider a unitary connection and a skew-Hermitian Higgs field on the trivial bundle $M\times \C^n$. We define a {\it\bfseries unitary connection} as a matrix-valued smooth map $A:TM\to \mathfrak{u}(n)$ which for fixed $x\in M$ is linear in $v\in T_x M$, and define a {\it\bfseries skew-Hermitian Higgs field} as a matrix-valued smooth map $\Phi:M\to \mathfrak{u}(n)$. The connection $A$ induces a covariant derivative which acts on sections of $M\times \C^n$ by $d_A:=d+A$. Saying $A$ is unitary means the following holds for the inner product of sections $s_{1},s_{2}$ of $M\times \C^n$
$$
d(s_1,s_2)=(d_As_1,s_2)+(s_1,d_A s_2).
$$
Pairs of unitary connections and skew-Hermitian Higgs fields $(A,\Phi)$ are very important in the Yang-Mills-Higgs theories, since they correspond to the most popular structure groups $U(n)$ or $SU(n)$, see \cite{Dun,HSW,MS,MW}.

On the boundary of $M$, we consider the set of inward and outward unit vectors defined as
\begin{align*}
\p_+ SM&=\{(x,v)\in SM:x\in\p M,\langle v,\nu(x)\rangle\ge0\},\\
\p_- SM&=\{(x,v)\in SM:x\in\p M,\langle v,\nu(x)\rangle\le0\},
\end{align*}
where $\nu$ is the unit inner normal to $\p M$. The magnetic geodesics entering $M$ can be parametrized by $\p_+SM$. We say that a magnetic system $(M,g,\Omega)$ is {\it\bfseries non-trapping} if for any $(x,v)\in SM$ the time $\tau_{+}(x,v)$ when the magnetic geodesic $\gamma_{x,v}$, with $x=\gamma_{x,v}(0)$, $v=\dot\gamma_{x,v}(0)$, exits $M$ is finite.  In particular, simple magnetic systems are non-trapping \cite{DPSU}.

Let $\mathbf G_\mu$ denote the generating vector field of the magnetic flow $\varphi_t$. Given $f\in C^\infty(SM,\C^n)$, consider the following transport equation for $u:SM\to \C^n$
$$
\mathbf G_\mu u+Au+\Phi u=-f\text{ in }SM,\quad u\big|_{\p_- SM}=0.
$$
Here $A$ and $\Phi$ act on functions on $SM$ by matrix multiplication. This equation has a unique solution $u^f$, since on any fixed magnetic geodesic the transport equation is a linear system of ODEs with zero initial condition.
\begin{Definition}{\rm
The \emph{\bfseries attenuated magnetic ray transform} of $f\in C^\infty(SM,\C^n)$, with attenuation determined by a unitary connection $A:TM\to \mathfrak u(n)$ and a skew-Hermitian Higgs field $\Phi:M\rightarrow\mathbb{C}^{n}$, is given by
$$
I_{A,\Phi} f:=u^f\big|_{\p_+ SM}.
$$}
\end{Definition}

It is clear that a general function $f\in C^\infty(SM,\C^n)$ cannot be determined by its attenuated magnetic ray transform, since $f$ depends on more variables than $I_{A,\Phi} f$. Moreover, one can easily see that the functions of the following type are always in the kernel of $I_{A,\Phi}$
\begin{equation}\label{general kernel}
(\mathbf G_\mu+A+\Phi)u,\quad u\in C^\infty(SM,\C^n),\ u\big|_{\p(SM)}=0.
\end{equation}
However, in applications one often needs to invert the transform $I_{A,\Phi}$ acting on functions on $SM$ arising from symmetric tensor fields. Further, we will consider this particular case.

Let $f=f_{i_1\cdots\,i_m}dx^{i_1}\otimes\cdots\otimes dx^{i_m}$ be a $\C^n$-valued, smooth symmetric $m$-tensor field on $M$. Then a tensor field induces a smooth function $f_m\in C^\infty(SM,\C^n)$ by
$$
f_m(x,v)=f_{i_1\cdots\,i_m}(x)\,v^{i_1}\cdots v^{i_m}.
$$
We denote by $C^\infty(S_m(M),\C^n)$ the bundle of smooth $\C^n$-valued, (covariant) symmetric $m$-tensor fields on $M$. When $m=1$, we also use the notation $C^\infty(\Lambda^1(M),\C^n)$. By $I_{A,\Phi}^m$ we denote the following operator
$$
I^m_{A,\Phi}f:=I_{A,\Phi}f_m,\quad f\in C^\infty(S_m(M),\C^n).
$$
We will frequently identify the tensor field $f\in C^\infty(S_m(M),\C^n)$ with the corresponding function $f_m\in C^\infty(SM,\C^n)$. Note that $I^{0}_{A,\Phi}$ has domain $C^{\infty}(M,\mathbb{C}^{n})$.

\subsection{Structure of the paper}
In this paper we give a description for the functions in $C^\infty(\p_+ SM,\C^n)$ which are in the range of $I_{A,\Phi}$. For the description we use the following boundary data: the scattering relation (see Section~\ref{2.1}), the scattering data of the pair $(A,\Phi)$ (see Section~\ref{2.2}) and the fibrewise Hilbert transform at the boundary (see Section~\ref{2.3}). The structure of the paper is as follows. In Section~\ref{2} we recall some facts and definitions from \cite{Ain,PSU} that will be used in our paper. A description of the range of the ray transform acting on higher order tensors is given in Section~\ref{3}, based on Theorem~\ref{main}. In Section~\ref{4} we discuss the surjectivity properties of the adjoint of the attenuated ray transform. Finally, in Section~\ref{5} we give the proof of Theorem~\ref{main}.


\section{Preliminaries}\label{2}
\subsection{Scattering relation}\label{2.1}
Let $(M,g,\alpha)$ be a simple magnetic system. For $(x,\xi)\in SM$, we denote by $\gamma_{x,\xi}$ the magnetic geodesic on $M$ such that $\gamma_{x,\xi}(0)=x$,
$\dot\gamma_{x,\xi}(0)=\xi$. By $\tau_+(x,\xi)$ and $\tau_-(x,\xi)$ we denote the nonnegative and nonpositive times, respectively, when $\gamma_{x,\xi}$ exits $M$. Simplicity of $(M,g,\alpha)$ implies that $\tau_+$ and $\tau_-$ are continuous on $SM$ and smooth on $SM\setminus S(\p M)$. Furthermore, from \cite[Lemma~2.3]{DPSU} we also have that $\tau_+\big|_{\p_+ SM}$ and $\tau_-\big|_{\p_- SM}$ are smooth.

The {\it\bfseries scattering relation} $\mathcal S:\p_+ SM\to\p_-SM$ is defined as
$$
\mathcal S(x,\xi):=(\varphi_{\tau_+(x,\xi)}(x,\xi))=(\gamma_{x,\xi}(\tau_+(x,\xi)),\dot\gamma_{x,\xi}(\tau_+(x,\xi))).
$$
From the above comments on $\tau_+$, we conclude that the scattering relation $\mathcal S$ is a smooth map.

For a given $w\in C^\infty(\p_+SM,\C^n)$, the transport equation
$$
\mathbf G_\mu u=0\text{ in }SM,\quad u\big|_{\p_+SM}=w
$$
has the solution $u=w_\psi:=w\circ\mathcal S^{-1}\circ\psi$ where $\psi:SM\to\p_-SM$ is the end point map $\psi(x,\xi):=\varphi_{\tau_+(x,\xi)}(x,\xi)$.

\subsection{Scattering data of a unitary connection and skew-Hermitian Higgs field}\label{2.2}
Let $U_{A,\Phi}:SM\to U(n)$ be the unique solution of the transport equation
$$
\mathbf G_\mu U_{A,\Phi}+(A+\Phi)U_{A,\Phi}=0\text{ in }SM,\quad U_{A,\Phi}\big|_{\p_+ SM}=\id.
$$
The map $C_{A,\Phi}:\p_-SM\to U(n)$ defined by $C_{A,\Phi}:=U_{A,\Phi}\big|_{\p_- SM}$ is called the {\it\bfseries scattering data} of the pair $(A,\Phi)$. If $G:M\to U(n)$ is a smooth map such that $G\big|_{\partial M}=\id$, then it is not difficult to check that the pair $(G^{-1}dG+G^{-1}AG,\,G^{-1}\Phi G)$ has the same scattering data. It was proved in \cite{Ain} that $(A,\Phi)$ can be determined by the scattering data $C_{A,\Phi}$ up to such a gauge equivalence.

Now, for a given $w\in C^\infty(\p_+ SM,\C^n)$ consider the unique solution $w^\sharp:SM\to \C^n$ to the transport equation
$$
\mathbf G_\mu w^\sharp+Aw^\sharp+\Phi w^\sharp=0\text{ in }SM,\quad w^\sharp\big|_{\p_+ SM}=w.
$$
Observe that $w^\sharp(x,v)=U_{A,\Phi}(x,v)w_\psi(x,v)$. Using the scattering relation $\mathcal S$ and the scattering data $C_{A,\Phi}$, we introduce the operator
$$
Q_{A,\Phi}:C(\p_+SM,\C^n)\to C(\p(SM),\C^n)
$$
defined by
$$
Q_{A,\Phi}w(x,v):=\begin{cases}
w(x,v)\quad&(x,v)\in\p_+SM,\\
C_{A,\Phi}(x,v)\,(w\circ\mathcal S^{-1})(x,v)\quad&(x,v)\in\p_-SM.
\end{cases}
$$
Then clearly $w^\sharp\big|_{\p(SM)}=Q_{A,\Phi} w(x,v)$. The space of those $w$ for which $w^\sharp$ is smooth in $SM$ is denoted by
$$
\mathcal S_{A,\Phi}^\infty(\p_+SM,\C^n)=\{w\in C^\infty(\p_+ SM,\C^n):w^\sharp\in C^\infty(SM,\C^n)\}.
$$
This space was characterized in \cite[Lemma~4.2]{Ain} in terms of the operator $Q_{A,\Phi}$ as follows:
$$
\mathcal S_{A,\Phi}^\infty(\p_+SM,\C^n)=\{w\in C^\infty(\p_+ SM,\C^n):Q_{A,\Phi}w\in C^\infty(\partial(SM),\C^n)\}.
$$

Using the fundamental solution $U_{A,\Phi}$, we can also give an integral expression for the ray transform. Recall $I_{A,\Phi}f:=u^{f}\big|_{\partial_{+}SM}$ where $u^{f}$ is the unique solution to
$$
\mathbf G_\mu u+Au+\Phi u=-f\text{ in }SM,\quad u|_{\p_- SM}=0.
$$
Note that $U_{A,\Phi}^{-1}$ solves $\mathbf G_\mu U_{A,\Phi}^{-1}-U_{A,\Phi}^{-1}(A+\Phi)=0$. Therefore, $\mathbf G_\mu(U_{A,\Phi}^{-1}u^{f})=-U_{A,\Phi}^{-1}f$. Integrating from $0$ to $\tau_{+}(x,v)$ for $(x,v)\in \partial_{+}SM$ we obtain the following expression
$$
u^{f}(x,v)=\int_0^{\tau_{+}(x,v)}U^{-1}_{A,\Phi}(\varphi_t(x,v))f(\varphi_t(x,v))\,dt.
$$

\subsection{Geometry and Fourier analysis on \texorpdfstring{$SM$}{SM}}
Since $M$ is assumed to be oriented there is a circle action on the fibres of $SM$ with infinitesimal generator $V$ called the vertical vector field. Let $X$ denote the generator of the geodesic flow of $g$. We complete $X,V$ to a global frame of $T(SM)$ by defining the vector field $X_\perp:=[V,X]$, where $[\cdot,\cdot]$ is the Lie bracket for vector fields. It is easy to see that the generator of magnetic flow $\varphi_t$ can be expressed in terms of the global frame $(X,X_\perp,V)$ in the following form
$$
\mathbf G_\mu=X+\lambda V,
$$
where $\lambda$ is the unique function satisfying $\Omega=\lambda d\Vol_g$ with $d\Vol_g$ being the area form of $M$. 

For any two functions $u,v:SM\to\C^n$ define an inner product:
$$
\langle u,v\rangle=\int_{SM} (u,v)_{\C^n}\,d\Sigma^3,
$$
where $d\Sigma^3$ is the Liouville measure of $g$ on $SM$. The space $L^2(SM,\C^n)$ decomposes orthogonally as a direct sum
$$
L^2(SM,\C^n)=\bigoplus_{k\in\mathbb Z}E_k
$$
where $E_k$ is the eigenspace of $-iV$ corresponding to the eigenvalue $k$. Any function $u\in C^{\infty}(SM,\C^n)$ has a Fourier series expansion
$$
u=\sum_{k=-\infty}^{\infty}u_k,
$$
where $u_k\in \Omega_k:=C^{\infty}(SM,\C^n)\cap E_k$. A function on $SM$ whose Fourier coefficients are trivial whenever $\left|k\right|> m$ is said to be of {\it\bfseries degree} $m$.

\subsection{Some elliptic operators of Guillemin and Kazhdan}
Now we introduce the following first order elliptic operators due to Guillemin and Kazhdan \cite{GK}
$$
\eta_+,\eta_-:C^\infty(SM,\C^n)\to C^\infty(SM,\C^n)
$$
defined by
$$
\eta_+:=\frac{1}{2}(X+iX_\perp),\quad \eta_-:=\frac{1}{2}(X-iX_\perp).
$$
By the commutation relations $[-iV,\eta_+]=\eta_+$ and $[-iV,\eta_-]=-\eta_-$ we see that
$$
\eta_+:\Omega_k\to\Omega_{k+1},\quad\eta_-:\Omega_k\to\Omega_{k-1}.
$$
We will use these operators in the last two sections.

\subsection{Fibrewise Hilbert transform}\label{2.3}
An important tool in our approach is the {\it\bfseries fibrewise Hilbert transform} $\mathcal H:C^\infty(SM,\C)\to C^\infty(SM,\C)$, which we define in terms of Fourier coefficients as
$$
\mathcal H(u_k)=-\sgn(k)iu_k,
$$
where we use the convention $\sgn(0)=0$. Moreover, $H(u)=\sum_k\mathcal H(u_k)$. Note that
$$
(\id+i\mathcal H)u=u_0+2\sum_{k=1}^\infty u_k,\quad
(\id-i\mathcal H)u=u_0+2\sum_{k=-\infty}^{-1} u_k.
$$

The following commutator formula, which was derived by Pestov and Uhlmann in \cite{PU1} and generalized in \cite{Ain,PSU}, will play an important role.
\begin{equation}\label{[H,X]}
[\mathcal H,\mathbf G_\mu+A+\Phi]u=(X_\perp+\star A)u_0+\((X_\perp+\star A)u\)_0,\quad u\in C^\infty(SM,\C).
\end{equation}
This formula has been frequently used in recent works on inverse problems, see \cite{PSU2,PSU1,PSU,PU1,PU,SaU}.

\subsection{Range description}
Given $w\in C^\infty(\p_+ SM,\C^n)$ we define $w^\sharp$ to be the unique solution to transport equation
$$
\mathbf G_\mu w^\sharp+Aw^\sharp+\Phi w^\sharp=0,\quad w^\sharp|_{\p_+ SM}=w.
$$
Moreover, we define $\mathcal S^{\infty}_{A,\Phi}(\p_+ SM,\C^n)$ to be the set of all $w\in C^\infty(\p_+ SM,\C^n)$ such that $w^\sharp$ is smooth.

We introduce the operator $B_{A,\Phi}:C(\p(SM),\C^n)\to C(\p_+SM,\C^n)$ defined by
$$
B_{A,\Phi}a:=[(C_{A,\Phi}^{-1}a)\circ \mathcal S-a]\big|_{\p_+ SM}.
$$
Let $a$ be a smooth function and $f=(\mathbf G_\mu+A+\Phi)a$. Suppose that $u^f$ solves
$$
\mathbf G_\mu u+(A+\Phi)u=-f
$$
with $u\big|_{\p_- SM}=0$. Then clearly $\mathbf G_\mu(u^f+a)+(A+\Phi)(u^f+a)=0$. Since $(u^f+a)\big|_{\p_- SM}=a\big|_{\p_- SM}$ we deduce that
$$
I_{A,\Phi}((\mathbf G_\mu+A+\Phi)a)=u^f\big|_{\p_+ SM}=[(C_{A,\Phi}^{-1}a)\circ\mathcal S-a]\big|_{\p_+ SM}=B_{A,\Phi}(a\big|_{\p(SM)}).
$$

Next we introduce the operator $P:\mathcal S_{A,\Phi}^\infty(\p_+ SM,\C^n)\to C^\infty(\p_+ SM,\C^n)$ defined by $P_{A,\Phi}:=B_{A,\Phi}\mathcal H Q_{A,\Phi}$. Clearly the operator $P$ is completely determined by the scattering relation $\mathcal S$ and scattering data $C_{A,\Phi}$.

Recall a connection $A$ induces an operator $d_A$, acting on $\C^n$-valued differential forms on $M$ by the formula $d_A\alpha=d\alpha+A\wedge\alpha$. By $d_A^*$ we denote the dual of $d_A$ with respect to $L^2$-norm on the space of forms. Then it is not difficult to check that $d^*_A=-\star d_A\star$. We use the notation $\mathfrak H_A$ for the space of all 1-forms $\eta$ with $d_A\eta=d_A^*\eta=0$ and $\jmath^*\eta=0$ where $\jmath:\p M\to M$ is the inclusion map. The elements of this space are called {\it\bfseries $A$-harmonic forms}. Note that $\mathfrak H_A$ is a finite dimensional space, since the equations defining $\mathfrak H_A$ are an elliptic system with regular boundary condition, see \cite[Section~5.11]{T}. Since $M$ is a disk, we have $\mathfrak H_A=0$ whenever $A=0$.

We can now state our main result.

\begin{Theorem}\label{main}
Let $(M,g,\alpha)$ be a two-dimensional simple magnetic system, $A$ a unitary connection and $\Phi$ a skew-Hermitian Higgs field. Then a function $u\in C^\infty(\p_+ SM,\C^n)$ belongs to the range of $I^0_{A,\Phi}+I^1_{A,\Phi}$ if and only if
$$
u=P_{A,\Phi}w+I^1_{A,\Phi}\eta
$$
for some $w\in\mathcal S_{A,\Phi}^\infty(\p_+ SM,\C^n)$ and for some $\eta\in\mathfrak H_A$.
\end{Theorem}

Theorem~\ref{main} was proved by Paternain, Salo and Uhlmann \cite{PSU2} in the case of the geodesic flow. This was used to give a description of the range of unattenuated geodesic ray transform acting on higher order tensors. For the characterization of the range of the ray transform on higher order tensors our main result is stated in Section~\ref{3}: Theorem~\ref{main2-tensor}.\\
\indent The crucial difficulty when one deals with a magnetic field or Higgs field is the fact that the concomitant transport equation couples different Fourier components. Even if one restricts oneself to the geodesic case, by adding a Higgs field this difficulty already presents an obstacle to the approach of \cite{PSU2}. The key idea that is utilized to overcome this, and which represents the major contribution of this paper, is a result on the ``simultaneous'' surjectivity of the adjoints of the ray transform $I^{0}_{A,\Phi}$ and $I^{1}_{A,\Phi}$ - this is made precise in Theorem~\ref{surjectivity of I_A^*}. This theorem relies upon considering the ray transform restricted to domains on which it is genuinely injective, namely $\Omega_{0}\oplus\Omega_{1}$ and $\Omega_{-1}\oplus\Omega_{0}$, and then patching solutions together. One should not directly consider the ray transform restricted to $\Omega_{-1}\oplus\Omega_{0}\oplus\Omega_{1}$ because the kernel of the ray transform on this domain has a natural obstruction, and in particular, is not identically $0$.

\section{Range characterization for higher order tensors}\label{3}
Let $\kappa$ denote the canonical line bundle of $M$, whose complex structure is that induced by its metric $g$. For $k\in\mathbb{N}$ we denote by $\Gamma(M,\kappa^{\otimes k})$ the set of $k$-th tensor power of canonical line bundle. It was explained in \cite[Section~2]{PSU3} the set $\Gamma(M,\kappa^{\otimes k})$ can be identified with $\Omega_k$. Roughly speaking, for a given $\xi\in \Gamma(M,\kappa^{\otimes k})$ we obtain a corresponding function on $\Omega_k$ via the one-to-one map $SM\ni (x,v)\mapsto \xi_x(v^{\otimes k})$.

Since $M$ is a disk, there is a nonvanishing $\xi\in\Gamma(M,\kappa)$. Define a function $h:SM\to S^1$ by setting $h(x,v):=\xi_x(v)/|\xi_x(v)|$, and hence $h\in\Omega_1$. For the description of the range of $I_{A,\Phi}$ on higher order tensors we use the following unitary connection $A_{h}:=-h^{-1}Xh\id$ and skew-Hermitian Higgs field $\Phi_{\lambda}:=-i\lambda\id$. Then it is easy to see that $A_{h}$ and $\Phi_{\lambda}$ satisfy
$$
-h^{-1}\mathbf G_\mu h\id=A_{h}+\Phi_{\lambda}.
$$

We start with characterizing the range of attenuated magnetic ray transform $I_{A,\Phi}$ restricted to $\Omega_{m-1}\oplus\Omega_m\oplus\Omega_{m+1}$:
$$
I^\pm_{m,A,\Phi}:=I_{A,\Phi}|_{\Omega_{m-1}\oplus\Omega_m\oplus\Omega_{m+1}}:\Omega_{m-1}\oplus\Omega_m\oplus\Omega_{m+1}\to C^\infty(\p_+ SM,\C).
$$
Consider $f+f'+f''\in \Omega_{m-1}\oplus\Omega_m\oplus\Omega_{m+1}$. Let $u$ be the solution of $(\mathbf G_\mu+A+\Phi)u=-(f+f'+f'')$ with $u|_{\p_-SM}=0$. Then $h^{-m}u$ solves
$$
(\mathbf G_\mu+A+\Phi-mA_{h}-m\Phi_{\lambda}) h^{-m}u=-h^{-m}(f+f'+f''),\quad h^{-m}u|_{\p_-SM}=0.
$$
Since $f+f'+f''\in \Omega_{m-1}\oplus\Omega_m\oplus\Omega_{m+1}$ and $h^m\in \Omega_m$, we have $h^{-m}(f+f'+f'')\in\Omega_{-1}\oplus\Omega_0\oplus\Omega_{1}$. Therefore
\begin{align}
I^0_{A-mA_{h},\Phi-m\Phi_{\lambda}}(h^{-m}f')&+I^1_{A-mA_{h},\Phi-m\Phi_{\lambda}}(h^{-m}f+h^{-m}f'')\nonumber\\
                                             &\ \ \ \ \ \ \ \ =(h^{-m}|_{\p_+SM})I^\pm_{m,A,\Phi} (f+f'+f'').\label{transition-formula}
\end{align}
The range of the left hand side of \eqref{transition-formula} was described in Theorem~\ref{main}. Thus we directly conclude the following result.

\begin{Theorem}\label{transition-theorem}
Let $(M,g,\alpha)$ be a simple two-dimensional magnetic system, $A$ a unitary connection and $\Phi$ a skew-Hermitian Higgs field. Then a function $u\in C^\infty(\p_+ SM,\C)$ belongs to the range of $I^\pm_{m,A,\Phi}$ if and only if
$$
u=(h^m|_{\p_+SM})\(P_{A-mA_{h},\Phi-m\Phi_{\lambda}}(w)+I^1_{A-mA_{h},\Phi-m\Phi_{\lambda}}\eta\)
$$
for some $w\in\mathcal S_{A-mA_{h},\Phi-m\Phi_{\lambda}}^\infty(\p_+ SM,\C)$ and $\eta\in\mathfrak H_{A-mA_{h}}$.
\end{Theorem}

Recall that, according to \cite[Section~2]{PSU1}, there is a one-to-one correspondence between $C^\infty(S_m(M),\C^n)$ and a subspace of the set of functions on $SM$ of the form $f=\sum_{k=-m}^m f_k$, $f_k\in\Omega_k$. For functions of finite degree such as this, we have
$$
\mathcal I_{A,\Phi}(f)=\sum_{k=-\left[(m+1)/3\right]}^{\left[(m+1)/3\right]} I^\pm_{3k,A,\Phi}\(f_{3k-1}+f_{3k}+f_{3k+1}\).
$$
Therefore, using this and Theorem~\ref{transition-theorem}, we obtain our second main result of the current paper.

\begin{Theorem}\label{main2-tensor}
Let $(M,g,\alpha)$ be a simple two-dimensional magnetic system, $A$ a unitary connection and $\Phi$ a skew-Hermitian Higgs field. Then a function $u\in C^\infty(\p_+ SM,\C^n)$ belongs to the range of 
$$
I_{A,\Phi}\big|_{\Omega_{-m}\oplus\cdots\oplus\,\Omega_{m}}
$$ if and only if there are $w_{3k}\in\mathcal S_{A-3kA_{h},\Phi-3k\Phi_{\lambda}}^\infty(\p_+ SM,\C)$ and $\eta_{3k}\in\mathfrak H_{A-3kA_{h}}$ for all $k=-[(m+1)/3],...1,\dots,[(m+1)/3]$ such that
$$
u=\sum_{k=-\left[(m+1)/3\right]}^{\left[(m+1)/3\right]}\(h^{3k}|_{\p_+SM})\left(P_{A-3kA_{h},\Phi-3k\Phi_{\lambda}}(w_{3k})+I^1_{A-3kA_{h},\Phi-3k\Phi_{\lambda}}\eta_{3k}\right)\).
$$
\end{Theorem}

\section{Surjectivity properties of \texorpdfstring{$I_{A,\Phi}$}{I*}}\label{4}
Let $d\Sigma^2$ be the volume form on $\p(SM)$. In the space of $\C^n$-valued functions on $\p_+ SM$ define the inner product
$$
\<h,h'\>_{\mu}=\int_{\p_+ SM}(h,h')_{\C^n}\,d\mu(x,v)
$$
where $d\mu(x,v)=\< v,\nu(x)\>d\Sigma^2(x,v)$. Denote the corresponding Hilbert space by $L_\mu^2(\p_+ SM,\C^n)$. As in \cite{PSU2}, using the integral representation for $I_{A,\Phi}$ and Santal\'o formula \cite[Lemma~A.8]{DPSU}, one can show that $I_{A,\Phi}$ can be extended to a bounded operator $I_{A,\Phi}:L^2(SM,\C^n)\to L_\mu^2(\p_+ SM,\C^n)$. In \cite{PSU2} it is shown that 
\begin{equation}\label{adjointformula}
\left(I_{A,\Phi}\right)^{*}h=\left(U_{A,\Phi}^{-1}\right)^{*}h_{\psi}
\end{equation}
The formula there is stated for geodesic flows with no attenuating Higgs field, however, the proof extends immediately to our case. Moreover, if $i_{k}:E_{k}\rightarrow L^2(SM,\C^n)$ denotes the usual inclusion, then 
\begin{equation}
\left(I_{A,\Phi}\circ i_{k}\right)^{*}h=\left(\left(U_{A,\Phi}^{-1}\right)^{*}h_{\psi}\right)_{k},
\end{equation}
where the final subscript denotes orthogonal projection onto $E_{k}$. Since we only deal with unitary connections and skew-Hermitian Higgs fields we have $\left(U_{A,\Phi}^{-1}\right)^{*}=U_{A,\Phi}$, and the formulas simplify. The identifications of $E_{0}$ and $L^{2}(M,\mathbb{C}^{n})$, and $E_{-1}\oplus E_{1}$ and $L^2(\Lambda^1(M),\C^n)$ mean that the adjoints we are concerned with differ from those above by the following constants 
\begin{align}
\left(I^{0}_{A,\Phi}\right)^{*}h &= 2\pi\left(U_{A,\Phi}h_{\psi}\right)_{0}\label{Remark1}\\ 
\left(I^{1}_{A,\Phi}\right)^{*}h &=\pi\left(U_{A,\Phi}h_{\psi}\right)_{-1} + \pi\left(U_{A,\Phi}h_{\psi}\right)_{1}\label{Remark2}
\end{align}
see \cite[Remark 5.2]{PSU2}. Below, we give an explicit calculation for the adjoint of 
$$
I_{A,\Phi}^{1}:L^2(\Lambda^1(M),\C^n)\rightarrow L_\mu^2(\p_+ SM,\C^n).
$$
To this end let $\beta$ be a $\mathbb{C}^{n}$-valued $1$-form and $h\in L^{2}_{\mu}(\partial_{+}(SM),\mathbb{C}^{n}).$
\begin{align*}
&\left\langle I^{1}_{A,\Phi}\beta,h\right\rangle_{\mu}\\
&= \int_{\partial_{+}(SM)}\\
&\quad\quad\quad\left\langle\int_{0}^{\tau_{+}(x,v)}\!\!\!\!\!\!\!U_{A,\Phi}^{-1}(\varphi_{t}(x,v))\beta(\varphi_{t}(x,v))dt,h(x,v) \right\rangle_{\mathbb{C}^{n}}\!\!d\mu(x,v)\\
&= \int_{\partial_{+}(SM)} \int_{0}^{\tau_{+}(x,v)}\\
&\quad\quad\quad\left\langle U_{A,\Phi}^{-1}(\varphi_{t}(x,v))\beta(\varphi_{t}(x,v)),h_{\psi}(\varphi_{t}(x,v)) \right\rangle_{\mathbb{C}^{n}}dtd\mu(x,v)\\
&= \int_{SM} \left\langle U_{A,\Phi}^{-1}(x,v)\beta(x,v),h_{\psi}(x,v) \right\rangle_{\mathbb{C}^{n}}d\Sigma^{3}(x,v)\\
&= \int_{SM} \left\langle \beta(x,v),\left(U_{A,\Phi}^{-1}(x,v)\right)^{*}h_{\psi}(x,v) \right\rangle_{\mathbb{C}^{n}}d\Sigma^{3}(x,v)\\
&= \int_{M} \int_{S_{x}M} \left\langle \beta(x,v),\left(U_{A,\Phi}^{-1}(x,v)\right)^{*}h_{\psi}(x,v) \right\rangle_{\mathbb{C}^{n}}dS_{x}(v)d\mbox{Vol}_{g}(x)\\
&= \int_{M} \int_{S_{x}M} \left\langle \beta_{i}(x)v^{i},\left(U_{A,\Phi}^{-1}(x,v)\right)^{*}h_{\psi}(x,v) \right\rangle_{\mathbb{C}^{n}}dS_{x}(v)d\mbox{Vol}_{g}(x)\\
&= \int_{M} \left\langle \beta_{i}(x),\int_{S_{x}M}v^{i}\left(U_{A,\Phi}^{-1}(x,v)\right)^{*}h_{\psi}(x,v)dS_{x}(v) \right\rangle_{\mathbb{C}^{n}}d\mbox{Vol}_{g}(x)\\
\end{align*}
Hence,
\begin{align*} 
&(I^{1}_{A,\Phi})^{*}(h)(x)\\
&= \int_{S_{x}M}v^{1}U_{A,\Phi}(x,v)h_{\psi}(x,v)dS_{x}(v)\epsilon_{1}\ + \int_{S_{x}M}v^{2}U_{A,\Phi}(x,v)h_{\psi}(x,v)dS_{x}(v)\epsilon_{2},
\end{align*}
where $\left\{e_{1},e_{2}\right\}$ is an orthonormal frame for the tangent bundle and $\left\{\epsilon_{1},\epsilon_{2}\right\}$ is its dual frame. The aim of the rest of this section is to prove the following analogue of \cite[Theorems 5.4 \& 5.5]{PSU2}.
\begin{Theorem}\label{surjectivity of I_A^*}
Let $(M,g,\alpha)$ be a two-dimensional simple magnetic system, $A$ a unitary connection and $\Phi$ a skew-Hermitian Higgs field. Suppose one is given $f\in C^{\infty}(M,\mathbb{C}^{n})$ and a $\mathbb{C}^{n}$-valued $1$-form $\omega$. Then there exists $w\in\mathcal{S}^{\infty}_{A,\Phi}(\partial_{+}(SM),\mathbb{C}^{n})$ such that $(I^{0}_{A,\Phi})^{*}(w)=f$ and $(I^{1}_{A,\Phi})^{*}(w)=\omega$ iff the following equality is satisfied $d_{A}^{*}\omega=2\Phi f$.
\end{Theorem}

\subsection{Surjectivity of adjoints}
Recall that for our oriented Riemannian surface $(M,g)$ we have 
$$L^{2}(SM,\mathbb{C}^{n})=\bigoplus_{k\in\mathbb{Z}}E_{k},$$
where $E_{k}$ is the $k$-eigenspace for $-iV$. We also have $\Omega_{k}:=C^{\infty}(SM,\mathbb{C}^{n})\cap E_{k}.$ We introduce here some more notation to specify such spaces for manifolds other than $M$.

Let $W$ be an oriented Riemannian surface. We have the following decomposition
$$L^{2}(SW,\mathbb{C}^{n})=\bigoplus_{k\in\mathbb{Z}}E^{W}_{k},$$
where $E^{W}_{k}$ is the $k$-eigenspace for $-iV$.
We define $C^{\infty}(W_{k}\oplus W_{l},\mathbb{C}^{n}):=C^{\infty}(SW,\mathbb{C}^{n})\cap \left(E^{W}_{k}\oplus E^{W}_{l}\right)$. Note that $C^{\infty}(M_{k}\oplus M_{l},\mathbb{C}^{n})=\Omega_{k}\oplus\Omega_{l}$.

For this section $I_{A,\Phi}^{0,1}$ will denote the ray transform with domain $E_{0}\oplus E_{1}$. We have $$N_{A,\Phi}:=I^{*}_{A,\Phi}I_{A,\Phi}:L^{2}(SM,\mathbb{C}^{n})\rightarrow L^{2}(SM,\mathbb{C}^{n}),$$
and similarly
$$N^{0,1}_{A,\Phi}:=(I^{0,1}_{A,\Phi})^{*}I^{0,1}_{A,\Phi}:E_{0}\oplus E_{1}\rightarrow E_{0}\oplus E_{1}.$$
Below we extend $M$ to a slightly larger manifold $\widetilde{M}$. The notation $\tilde{I}_{A,\Phi}^{0,1}$ will denote the ray transform with domain $E^{\widetilde{M}}_{0}\oplus E^{\widetilde{M}}_{1}$, and $\tilde{N}^{0,1}_{A,\Phi}:=(\tilde{I}^{0,1}_{A,\Phi})^{*}\tilde{I}^{0,1}_{A,\Phi}$.

For the proof of Theorem~\ref{surjectivity of I_A^*} we need the following:

\begin{Proposition}\label{surj for m=0}
For given $f\in \Omega_{0}$ and $\eta\in\Omega_{1}$, there is $u\in C^\infty(SM,\C^n)$ such that $(\mathbf G_\mu+A+\Phi)u=0$ and $u_0=f$, $u_{1}=\eta$.
\end{Proposition}

\begin{proof}

We follow ideas from \cite{DU}. Embed $M$ into the interior of a compact surface with boundary $\widetilde{M}$, extend $g$ and $\alpha$ to $\widetilde{M}$, choosing $(\widetilde{M},g,\alpha)$ to be sufficiently close to $(M,g,\alpha)$ so that it remains simple. We also assume that $A$ and $\Phi$ are extended unitarily to $\widetilde{M}$, and we'll keep the same notations for the extensions. Let $r_{M}$ denote the restriction operator from $\widetilde{M}$ to $M$. The following is true:

\begin{Lemma} The operator 
$$r_{M}\tilde{N}^{0,1}_{A,\Phi}:C^{\infty}_{c}(\widetilde{M}^{\operatorname{int}}_{0}\oplus \widetilde{M}^{\operatorname{int}}_{1},\mathbb{C}^{n})\rightarrow C^{\infty}(M_{0}\oplus M_{1},\mathbb{C}^{n})$$
is surjective.
\label{lemma2.2}
\end{Lemma}

Assuming this lemma we prove Proposition~\ref{surj for m=0}. Suppose $[f,\eta]\in C^{\infty}(M_{0}\oplus M_{1},\mathbb{C}^{n})$. By Lemma~\ref{lemma2.2} there exists $[h,\beta]\in C^{\infty}_{c}(\widetilde{M}^{\operatorname{int}}_{0}\oplus \widetilde{M}^{\operatorname{int}}_{1},\mathbb{C}^{n})$ such that
$$[f,\eta]=r_{M}\tilde{N}^{0,1}_{A,\Phi}[h,\beta]=r_{M}(\tilde{I}^{0,1}_{A,\Phi})^{*}\tilde{I}^{0,1}_{A,\Phi}[h,\beta].$$
Recall, $\widetilde{U}_{A,\Phi}$ solves
$$(\mathbf G_\mu+A+\Phi)\widetilde{U}_{A,\Phi}=0,\ \ \widetilde{U}_{A,\Phi}\big|_{\partial_{+}(S\widetilde{M})}=\operatorname{Id}.$$
Now define 
$$\tilde{w}(x,v):=\int^{\tilde{\tau}_{+}(x,v)}_{\tilde{\tau}_{-}(x,v)}\widetilde{U}^{-1}_{A,\Phi}(\gamma_{x,v}(t),\dot{\gamma}_{x,v}(t))\left[h(\gamma_{x,v}(t))+\beta_{i}(\gamma_{x,v}(t))\dot{\gamma}^{i}_{x,v}(t)\right]dt.$$
Here $\tilde{\tau}_{+}(x,v)$ is the unique time $t\geq 0$ when the geodesic $\gamma_{x,v}(t)$ hits the boundary $\partial \widetilde{M}$, and  $\tilde{\tau}_{-}(x,v)$ is the unique time $t\leq 0$ when the geodesic $\gamma_{x,v}(t)$ hits the boundary $\partial \widetilde{M}$. Define $w':=\tilde{w}\big|_{\partial_{+}(SM)}$ and note that $\tilde{w}\in C^{\infty}(S\widetilde{M},\mathbb{C}^{n}).$ Now from the definition one may check that 
$$\tilde{I}^{0,1}_{A,\Phi}[h,\beta]=\tilde{w}\big|_{\partial_{+}(S\widetilde{M})}.$$
Using the formula for the adjoint (\ref{adjointformula}) we obtain, 
$$r_{M}(\tilde{I}^{0,1}_{A,\Phi})^{*}\tilde{I}^{0,1}_{A,\Phi}[h,\beta]=\left[\widetilde{U}_{A,\Phi}\tilde{w}\right]_{0,1}\Big|_{SM}.$$
Therefore,
\begin{align*} 
[f,\eta]   =&\, r_{M}(\tilde{I}^{0,1}_{A,\Phi})^{*}\tilde{I}^{0,1}_{A,\Phi}[h,\beta]\\
           =& \left[\widetilde{U}_{A,\Phi}\tilde{w}\right]_{0,1}\Big|_{SM}\\
           =&\, (I^{0,1}_{A,\Phi})^{*}\left(\widetilde{U}_{A,\Phi}\big|_{\partial_{+}(SM)}w'\right)
\end{align*}
By defining $w:=\widetilde{U}_{A,\Phi}\big|_{\partial_{+}(SM)}w'$ we have proven our theorem.
\end{proof}

\begin{Remark} The last step above requires one to understand that the two solutions to the transport equation on $SM$ and $S\widetilde{M}$ are related as follows
\begin{equation}
\widetilde{U}_{A,\Phi}(x,v)=U_{A,\Phi}(x,v)\left(\widetilde{U}_{A,\Phi}\big|_{\partial_{+}(SM)}\right)_{\psi}\!(x,v)\ \ \ \forall (x,v)\in SM.\label{eq: smalltobig} 
\end{equation}
To check this one first shows that $R:=U^{-1}_{A,\Phi}\widetilde{U}_{A,\Phi}$ is invariant under the flow. Furthermore,
$$\widetilde{U}_{A,\Phi}\big|_{\partial_{+}(SM)}=R\big|_{\partial_{+}(SM)}.$$
Therefore, $R=\left(\widetilde{U}_{A,\Phi}\big|_{\partial_{+}(SM)}\right)_{\psi},$ which is equivalent to (\ref{eq: smalltobig}).
\end{Remark}
Finally, we must prove Lemma~\ref{lemma2.2}.

\begin{proof}[Proof of Lemma~\ref{lemma2.2}] Step 1. Let $[f,\eta]\in C^{\infty}(\widetilde{M}_{0}\oplus \widetilde{M}_{1},\mathbb{C}^{n}).$ Suppose that $\tilde{I}^{0,1}_{A,\Phi}[f,\eta]=0$, this implies the existence of $u$ satisfying 
$$(\mathbf G_\mu+A+\Phi)u=-f-\eta,\ \ u\big|_{\partial(S\widetilde{M})}=0.$$
Moreover, from \cite{Ain}, $u\equiv u_{0}$, so $d_{A}u=-\eta$ and $\Phi u=-f$. Defining $\bar{\partial}_{A}:=d_{A}-i\star d_{A}$, we see that $\eta\in\Omega_{1}$ implies that $\bar{\partial}_{A}u=0$. Now because $\widetilde{M}$ has global coordinates we have that there exists a smooth map $F:\widetilde{M}\rightarrow GL(n,\mathbb{C})$ such that $F^{-1}\overline{\partial}(F\cdot)=\overline{\partial}_{A}(\cdot)$ (see \cite[Proposition 3.7]{Kb}), and so each component of $Fu\in C^{\infty}(\widetilde{M},\mathbb{C}^{n})$ is a holomorphic function. Thus, $Fu\big|_{\partial \widetilde{M}}=0$ implies that $u\equiv 0,$ and consequently $f=0$ and $\eta=0$.\\
\indent Step 2. We seek to compute the principal symbol of $\tilde{N}^{0,1}_{A,\Phi}$ and see that it is an elliptic pseudodifferential operator of order $-1$ in $\widetilde{M}^{\operatorname{int}}$.

Let $[f,\eta]\in E^{\widetilde{M}}_{0}\oplus E^{\widetilde{M}}_{1}$. Then we have $\tilde{N}^{0,1}_{A,\Phi}:E^{\widetilde{M}}_{0}\oplus E^{\widetilde{M}}_{1}\rightarrow E^{\widetilde{M}}_{0}\oplus E^{\widetilde{M}}_{1}$ and we introduce the following notation
$$\tilde{N}^{0,1}_{A,\Phi}[f,\eta]=\left[\tilde{N}^{00}_{A,\Phi}f+\tilde{N}^{01}_{A,\Phi}\eta, \tilde{N}^{10}_{A,\Phi}f+\tilde{N}^{11}_{A,\Phi}\eta\right].$$
The analysis of this operator is relegated to Section~\ref{4.3}, wherein it is shown that $\tilde{N}^{0,1}_{A,\Phi}$ is a pseudodifferential operator of order $-1$ in $\widetilde{M}^{\operatorname{int}}$ whose principal symbols are given by
\begin{align*}
\sigma_{P}\left(\tilde{N}^{00}_{A,\Phi}\right)(x,\xi) &= \mbox{diag}\left\{C\frac{1}{\left|\xi\right|}\right\},\\
\sigma_{P}\left(\tilde{N}^{01}_{A,\Phi}\right)(x,\xi) &= \mbox{diag}\left\{0\right\},\\
\sigma_{P}\left(\tilde{N}^{10}_{A,\Phi}\right)(x,\xi) &= \mbox{diag}\left\{0\right\},\\
\sigma_{P}\left(\tilde{N}^{11}_{A,\Phi}\right)\!(x,\xi) &= \mbox{diag}\left\{\frac{C}{2}\frac{1}{\left|\xi\right|}\right\},
\end{align*}
for some constant $C>0.$ Note, in particular, the operator is elliptic.\\
\indent Step 3. We achieve surjectivity on Sobolev spaces using the injectivity and ellipticity established above, closely following the proof of \cite[Lemma 4.5]{SaU} and \cite[Proposition 4.5]{Po}. $(M,g,\alpha)$ sits inside the larger simple magnetic system $(\widetilde{M},g,\alpha)$. Now embed $(\widetilde{M},g,\alpha)$ into a compact oriented Riemannian surface $W$, without boundary, with a magnetic field which restricts to $\alpha$. Extend all of the attenuations smoothly to $W$, keeping them unitary. Choose a finite atlas for $W$ consisting of open sets $\left\{U_{k}\right\}$ each of whose closure is a simple magnetic system, and choose a partition of unity $\left\{\phi_{k}\right\}$ subordinate to this atlas. Without loss of generality the open set of the first chart is given by $\widetilde{M}^{\operatorname{int}}$, and the corresponding function from the partition of unity $\phi_{1}$ is identically $1$ on $M$. We introduce the notation $H^{s}(W_{0}\oplus W_{1}):=H^{s}(SW)\cap \left(E^{W}_{0}\oplus E^{W}_{1}\right)$ for Sobolev spaces and $\mathcal{D}'(W_{0}\oplus W_{1})$ for distributions, that is, continuous linear functionals on $C^{\infty}(W_{0}\oplus W_{1},\mathbb{C}^{n})$. We will use similar notation for the analogous spaces for $M$ and $\widetilde{M}^{\operatorname{int}}$ below. Now define an operator $P$ with domain $\mathcal{D}'(W_{0}\oplus W_{1})$ as follows
$$P(h) = \sum_{k} \phi_{k} (N_{k})^{0,1}_{A,\Phi}\left(\phi_{k}h\right).$$ 
Note that this acts on distributions via duality. In the formula above, for each $k$, $(N_{k})^{0,1}_{A,\Phi}$ is the operator $(I^{0,1}_{A,\Phi})^{*}I^{0,1}_{A,\Phi}$ that is associated with the $k^{th}$ open set in the atlas. By construction $(N_{1})^{0,1}_{A,\Phi}=\tilde{N}^{0,1}_{A,\Phi}$. Each $(N_{k})^{0,1}_{A,\Phi}$ is an elliptic pseudodifferential operator of order $-1$, with symbol as given in Step 2. Therefore, $P$ is an elliptic pseudodifferential operator of order $-1$, moreover, it is a Fredholm operator from $H^{s}(W_{0}\oplus W_{1})$ into $H^{s+1}(W_{0}\oplus W_{1})$. We introduce the restriction operator $r_{M}:H^{s}(W_{0}\oplus W_{1})\rightarrow H^{s}(M_{0}\oplus M_{1})$ which is bounded and surjective. From the fact that $\phi_{1}\big|_{M}\equiv 1$, we know that $r_{M}P(h)=r_{M}\tilde{N}^{0,1}_{A,\Phi}(\phi_{1}h)$. This shows that the following operator 
$$r_{M}\tilde{N}^{0,1}_{A,\Phi}(\phi_{1}\cdot):C^{\infty}(W_{0}\oplus W_{1},\mathbb{C}^{n})\rightarrow C^{\infty}(M_{0}\oplus M_{1},\mathbb{C}^{n}).$$
is a continuous linear map between the above Fr\'{e}chet spaces. In order to show it's surjective it is sufficient to show its adjoint is injective, and its range is weak$^{*}$ closed \cite[Theorem 37.2]{Tr}. Let $h\in C^{\infty}(W_{0}\oplus W_{1},\mathbb{C}^{n})$ and $u\in \mathcal{D}'(M_{0}\oplus M_{1})$. We employ $i_{\widetilde{M}^{\operatorname{int}}}$ to denote the usual inclusion mapping from Sobolev spaces of $M_{0}\oplus M_{1}$ into Sobolev spaces of $\widetilde{M}^{\operatorname{int}}_{0}\oplus \widetilde{M}^{\operatorname{int}}_{1}$.
\begin{align*}
\left\langle \left(r_{M}\tilde{N}^{0,1}_{A,\Phi}(\phi_{1}\cdot)\right)^{*}u,h\right\rangle_{W} &= \left\langle u,\left(r_{M}\tilde{N}^{0,1}_{A,\Phi}\right)\phi_{1}h\right\rangle_{M}\\
&= \left\langle u,\left(r_{M}\phi_{1}\tilde{N}^{0,1}_{A,\Phi}\right)\phi_{1}h\right\rangle_{M}\\
&= \left\langle i_{\widetilde{M}^{\operatorname{int}}}u,\phi_{1}\tilde{N}^{0,1}_{A,\Phi}\phi_{1}h\right\rangle_{W}\\
&= \left\langle \phi_{1}\tilde{N}^{0,1}_{A,\Phi}\phi_{1}i_{\widetilde{M}^{\operatorname{int}}}u,h\right\rangle_{W}.
\end{align*}
Also,
\begin{align*}
\left\langle \left(r_{M}\tilde{N}^{0,1}_{A,\Phi}(\phi_{1}\cdot)\right)^{*}u,h\right\rangle_{W} &= \left\langle u,\left(r_{M}\tilde{N}^{0,1}_{A,\Phi}\right)\phi_{1}h\right\rangle_{M}\\
&= \left\langle u,r_{M}Ph\right\rangle_{M}\\
&= \left\langle Pi_{\widetilde{M}^{\operatorname{int}}}u,h\right\rangle_{W}.
\end{align*}
Since these equalities hold for all $h\in C^{\infty}(W_{0}\oplus W_{1},\mathbb{C}^{n})$ we conclude that 
$$\left(r_{M}\tilde{N}^{0,1}_{A,\Phi}(\phi_{1}\cdot)\right)^{*} = \phi_{1}\tilde{N}^{0,1}_{A,\Phi}\phi_{1}i_{\widetilde{M}^{\operatorname{int}}} = P\big|_{\mathcal{D}'(M_{0}\oplus M_{1})}.$$
Now suppose $\left(r_{M}\tilde{N}^{0,1}_{A,\Phi}(\phi_{1}\cdot)\right)^{*}u=0.$ Therefore, $\phi_{1}\tilde{N}^{0,1}_{A,\Phi}\phi_{1}i_{\widetilde{M}^{\operatorname{int}}}u=0$. By ellipticity we deduce that $u$ is smooth on a slightly bigger set than $M$, and since it is supported in $M$, we have that $i_{\widetilde{M}^{\operatorname{int}}}u$ is smooth and compactly supported. Also, $0=\phi_{1}\tilde{N}^{0,1}_{A,\Phi}\phi_{1}i_{\widetilde{M}^{\operatorname{int}}}u$ implies, $\tilde{I}^{0,1}_{A,\Phi}\phi_{1}i_{\widetilde{M}^{\operatorname{int}}}u=0$. Hence, $u\equiv 0$ by Step 1.\\
\indent It remains to show $(r_{M}\tilde{N}^{0,1}_{A,\Phi}(\phi_{1}\cdot))^{*} = P\big|_{\mathcal{D}'(M_{0}\oplus M_{1})}$ has weak$^{*}$ closed range. This will follow once we show that the range of $P\big|_{\mathcal{D}'(M_{0}\oplus M_{1})}$ is weak$^{*}$ closed in $H^{s}(W_{0}\oplus W_{1})$ for each $s\in\mathbb{R}$. This will hold iff the range of $P\big|_{\mathcal{D}'(M_{0}\oplus M_{1})}$ intersected with the unit ball of $H^{s}(W_{0}\oplus W_{1})$ is weak$^{*}$ closed \cite[Corollary 3, IV.25]{Bo}. Suppose $Pu$ lies in the above intersection where $u\in \mathcal{D}'(M_{0}\oplus M_{1})$. Since $P:H^{s-1}(W_{0}\oplus W_{1})\rightarrow H^{s}(W_{0}\oplus W_{1})$ the Bounded Inverse Theorem implies $\left\|u\right\|_{H^{s-1}}\leq C\left\|Pu\right\|_{H^{s}}\leq C$, where $C>0$. Denote the set of all such $u\in\mathcal{D}'(M_{0}\oplus M_{1})$ by $J$. Then $J$ is weak$^{*}$ compact in $\mathcal{D}'(M_{0}\oplus M_{1})$, and $P(J)$ intersected with the unit ball of $H^{s}(W_{0}\oplus W_{1})$ will be weak$^{*}$ compact, and thus weak$^{*}$ closed.\\  
\indent Step 4. Finally, given $h'\in C^{\infty}(M_{0}\oplus M_{1},\mathbb{C}^{n})$, from Step 3 we can find $h\in C^{\infty}(W_{0}\oplus W_{1},\mathbb{C}^{n})$ such that $r_{M}\tilde{N}^{0,1}_{A,\Phi}(\phi_{1}h)=h'$. Therefore, $\phi_{1}h$ maps to $h'$, and is compactly supported in $\widetilde{M}^{\operatorname{int}}_{0}\oplus \widetilde{M}^{\operatorname{int}}_{1}$, as required.
\end{proof}

\begin{Proposition}\label{surj for all m}
For given $f\in \Omega_{m}$ and $\eta\in\Omega_{m+1}$, there is $u\in C^\infty(SM,\C^n)$ such that $(\mathbf G_\mu+A+\Phi)u=0$ and $u_m=f$, $u_{m+1}=\eta$.
\end{Proposition}
\begin{proof}
Fix a non-vanishing $h\in\Omega_1$. As before, consider the unitary connection $A_{h}=-h^{-1}Xh\id$ and skew-Hermitian Higgs field $\Phi_{\lambda}=-i\lambda\id$ satisfying
$$
-h^{-1}\mathbf G_\mu h\id=A_{h}+\Phi_{\lambda}.
$$
Note that for any $a\in C^\infty(SM,\C^n)$ the following holds
\begin{equation}\label{8eq}
(\mathbf G_\mu+A+\Phi-mA_{h}-m\Phi_{\lambda})a=-h^{-m}((\mathbf G_\mu+A+\Phi)(h^m a)).
\end{equation}
By Proposition~\ref{surj for m=0} there exists $a\in C^\infty(SM,\C^n)$ such that
$$
(\mathbf G_\mu+A+\Phi-mA_{h}-m\Phi_{\lambda})a=0
$$
and $a_0=h^{-m}f$, $a_1=h^{-m}\eta$. Set $u:=h^ma$, then clearly $u_m=f$ and $u_{m+1}=\eta$. By equality \eqref{8eq}, we have $(\mathbf G_\mu+A+\Phi)u=0$, which finishes the proof.
\end{proof}

\begin{proof}[Proof of Theorem~\ref{surjectivity of I_A^*}]
Let $\omega$ be a smooth $\C^n$-valued $1$-form on $M$ and let $f\in C^\infty(M,\C^n)$. By Proposition~\ref{surj for all m}, there exist $u,u'\in C^\infty(SM,\C^n)$ such that
$$
\begin{cases}
(\mathbf G_\mu+A+\Phi)u=0,\\
u_{-1}=\omega_{-1},\quad u_0=f,
\end{cases}\quad\text{and}\qquad\begin{cases}
(\mathbf G_\mu+A+\Phi)u'=0,\\
u'_0=f,\quad u'_{1}=\omega_{1}.
\end{cases}
$$
Consider $w\in C^\infty(SM,\C^n)$ defined as
$$
w:=\sum_{k=-\infty}^{-1}u_k+\sum_{k=1}^\infty u'_k+f.
$$
It is clear that $w_0=f$ and $w_{-1}+w_1=\omega_{-1}+\omega_1=\omega$. Introduce the operators $\mu_\pm =\eta_\pm+A_{\pm 1}$. Then $X+A=\mu_-+\mu_+$. It is easy to check that
$$
(\mathbf G_\mu+A+\Phi)w=0.
$$
In particular, we have
\begin{equation}
\mu_+w_{-1}+\mu_-w_1+\Phi w_0=0. \label{patchingcondition}
\end{equation}
By \cite[Lemma~6.2]{PSU2}, this is equivalent to our assumption $d^*_A\omega=2\Phi f$. Solutions of the transport equation are unique once boundary data is specified, therefore,  $w\big|_{\p_+ SM}$ satisfies the requirement of our theorem. Conversely, given $w\in\mathcal{S}^{\infty}_{A,\Phi}(\partial_{+}(SM),\mathbb{C}^{n})$ such that $(I^{0}_{A,\Phi})^{*}(w)=f$ and $(I^{1}_{A,\Phi})^{*}(w)=\omega$, then, in particular, $(\ref{patchingcondition})$ must hold, which implies $d^*_A\omega=2\Phi f$.
\end{proof}

\subsection{Addendum on the Symbol Computation}\label{4.3}
Here we provide some details regarding the computation of the principal symbol in Lemma~\ref{lemma2.2}. Let $[f,\eta]\in E^{\widetilde{M}}_{0}\oplus E^{\widetilde{M}}_{1}$. Let $\left\{e_{1},e_{2}\right\}$ be a local orthonormal frame for the tangent bundle with dual frame $\left\{\epsilon_{1},\epsilon_{2}\right\}$. Suppose 
$$\eta=\eta_{\epsilon_{1}}\epsilon_{1}+\eta_{\epsilon_{2}}\epsilon_{2}.$$
Then 
$$V(\eta)=-\eta_{\epsilon_{1}}\epsilon_{2}+\eta_{\epsilon_{2}}\epsilon_{1}.$$
And,
\begin{align*}
\eta_{1} =& \frac{1}{2}(\eta-iV(\eta))\\
=& \frac{1}{2}(\eta_{\epsilon_{1}}\epsilon_{1}+\eta_{\epsilon_{2}}\epsilon_{2}+i\eta_{\epsilon_{1}}\epsilon_{2}-i\eta_{\epsilon_{2}}\epsilon_{1})\\
=& \frac{1}{2}(\eta_{\epsilon_{1}}-i\eta_{\epsilon_{2}})(\epsilon_{1}+i\epsilon_{2}).
\end{align*}

\begin{align*}
\eta_{-1} =& \frac{1}{2}(\eta+iV(\eta))\\
=& \frac{1}{2}(\eta_{\epsilon_{1}}\epsilon_{1}+\eta_{\epsilon_{2}}\epsilon_{2}-i\eta_{\epsilon_{1}}\epsilon_{2}+i\eta_{\epsilon_{2}}\epsilon_{1})\\
=& \frac{1}{2}(\eta_{\epsilon_{1}}+i\eta_{\epsilon_{2}})(\epsilon_{1}-i\epsilon_{2}).
\end{align*}

Thus, 
\begin{equation}
\eta\equiv\eta_{1}\Leftrightarrow\eta_{\epsilon_{1}}=-i\eta_{\epsilon_{2}}
\label{omega1criteria}
\end{equation}
Now our operator is as follows $\tilde{N}^{0,1}_{A,\Phi}:E^{\widetilde{M}}_{0}\oplus E^{\widetilde{M}}_{1}\rightarrow E^{\widetilde{M}}_{0}\oplus E^{\widetilde{M}}_{1}$ and as introduced previously we'll employ the notation
$$\tilde{N}^{0,1}_{A,\Phi}[f,\eta]=\left[\tilde{N}^{00}_{A,\Phi}f+\tilde{N}^{01}_{A,\Phi}\eta, \tilde{N}^{10}_{A,\Phi}f+\tilde{N}^{11}_{A,\Phi}\eta\right].$$
Note that since we are working on a surface a $1$-form will locally be represented by two functions, however, $1$-forms constrained to lie in $E^{\widetilde{M}}_{1}$ will locally be represented by only one function, thanks to (\ref{omega1criteria}). Below we use $$\tilde{I}^{(+1)}_{A,\Phi}:=\tilde{I}_{A,\Phi}\big|_{E^{\widetilde{M}}_{1}}\ \text{and}\ \tilde{I}^{(0)}_{A,\Phi}:=\tilde{I}_{A,\Phi}\big|_{E^{\widetilde{M}}_{0}}.$$
This notation is only used in this section. We also write $\varphi_{x,v}(t)=\varphi^{e_{1}}_{x,v}(t)e_{1}+\varphi^{e_{2}}_{x,v}(t)e_{2}$. Recall also that $\tilde{U}_{A,\Phi}$ is the unique matrix solution to 
$$(\mathbf G_\mu+A+\Phi)\tilde{U}_{A,\Phi}=0\ \text{in}\ S\widetilde{M},\ \ \tilde{U}_{A,\Phi}\big|_{\partial_{+}(S\widetilde{M})}=\operatorname{Id}.$$

\begin{align*}
\left(\tilde{N}^{00}_{A,\Phi}f\right)(x):=&\left((\tilde{I}^{(0)}_{A,\Phi})^{*}\tilde{I}^{(0)}_{A,\Phi}f\right)(x)\\
=&\int_{S_{x}\widetilde{M}}\tilde{U}_{A,\Phi}(x,v)\int^{\tau_{+}(x,v)}_{\tau_{-}(x,v)}\tilde{U}^{-1}_{A,\Phi}\left(\varphi_{x,v}(t)\right)f(\varphi_{x,v}(t))dtdS_{x}(v)\\
\left(\tilde{N}^{01}_{A,\Phi}\eta\right)(x):=&\left((\tilde{I}^{(0)}_{A,\Phi})^{*}\tilde{I}^{(+1)}_{A,\Phi}\eta\right)(x)\\ 
=&\int_{S_{x}\widetilde{M}}\tilde{U}_{A,\Phi}(x,v)\\
&\times\int^{\tau_{+}(x,v)}_{\tau_{-}(x,v)}\tilde{U}^{-1}_{A,\Phi}\left(\varphi_{x,v}(t)\right)\left(\eta_{\epsilon_{1}}(\varphi_{x,v}(t))\varphi^{e_{1}}_{x,v}(t)\right)dtdS_{x}(v)\\
+&\int_{S_{x}\widetilde{M}}\tilde{U}_{A,\Phi}(x,v)\\
&\times\int^{\tau_{+}(x,v)}_{\tau_{-}(x,v)}\tilde{U}^{-1}_{A,\Phi}\left(\varphi_{x,v}(t)\right)\left(i\eta_{\epsilon_{1}}(\varphi_{x,v}(t))\varphi^{e_{2}}_{x,v}(t)\right)dtdS_{x}(v)\\
\left(\tilde{N}^{10}_{A,\Phi}f\right)(x):=&\left((\tilde{I}^{(+1)}_{A,\Phi})^{*}\tilde{I}^{(0)}_{A,\Phi}f\right)(x)\\
=&\int_{S_{x}\widetilde{M}}\frac{1}{2}\left(v^{e_{1}}-iv^{e_{2}}\right)\tilde{U}_{A,\Phi}(x,v)\\
&\times\int^{\tau_{+}(x,v)}_{\tau_{-}(x,v)}\tilde{U}^{-1}_{A,\Phi}\left(\varphi_{x,v}(t)\right)f(\varphi_{x,v}(t))dtdS_{x}(v)\\
\left(\tilde{N}^{11}_{A,\Phi}\eta\right)(x):=&\left((\tilde{I}^{(+1)}_{A,\Phi})^{*}\tilde{I}^{(+1)}_{A,\Phi}\eta\right)(x)\\
=&\int_{S_{x}\widetilde{M}}\frac{1}{2}\left(v^{e_{1}}-iv^{e_{2}}\right)\tilde{U}_{A,\Phi}(x,v)\\
&\times\int^{\tau_{+}(x,v)}_{\tau_{-}(x,v)}\tilde{U}^{-1}_{A,\Phi}\left(\varphi_{x,v}(t)\right)\left(\eta_{\epsilon_{1}}(\varphi_{x,v}(t))\varphi^{e_{1}}_{x,v}(t)\right)dtdS_{x}(v)\\
+&\int_{S_{x}\widetilde{M}}\frac{1}{2}\left(v^{e_{1}}-iv^{e_{2}}\right)\tilde{U}_{A,\Phi}(x,v)\\
&\times\int^{\tau_{+}(x,v)}_{\tau_{-}(x,v)}\tilde{U}^{-1}_{A,\Phi}\left(\varphi_{x,v}(t)\right)\left(i\eta_{\epsilon_{1}}(\varphi_{x,v}(t))\varphi^{e_{2}}_{x,v}(t)\right)dtdS_{x}(v)\\
\end{align*}

Using \cite[Lemma B.1]{DPSU} we see that $\tilde{N}^{0,1}_{A,\Phi}$ is a pseudodifferential operator in $\widetilde{M}^{\operatorname{int}}$ of order $-1$. (The lemma is stated for operators acting on scalar-valued functions, however, the proof of the lemma works for vector-valued functions as well \cite{St}.) In addition, the lemma guarantees that the principal symbol of each of the above operators is given by linear combinations of the principal symbols of the corresponding (unattenuated) operators which appear in \cite[Proposition 7.2]{DPSU}. In particular,
\begin{align*}
\sigma_{P}\left(\tilde{N}^{00}_{A,\Phi}\right)(x,\xi) &= \mbox{diag}\left\{C\frac{1}{\left|\xi\right|}\right\},\\
\sigma_{P}\left(\tilde{N}^{01}_{A,\Phi}\right)(x,\xi) &= \mbox{diag}\left\{0\right\},\\
\sigma_{P}\left(\tilde{N}^{10}_{A,\Phi}\right)(x,\xi) &= \mbox{diag}\left\{0\right\},
\end{align*}
\begin{align*}
\sigma_{P}&\left(\tilde{N}^{11}_{A,\Phi}\right)(x,\xi)\\
&= \mbox{diag}\left\{\frac{C}{2}\left(\frac{1}{\left|\xi\right|}-\frac{\xi^{\epsilon_{1}}\xi^{\epsilon_{1}}}{\left|\xi\right|^{3}}\right) - \frac{iC}{2}\left(\frac{\xi^{\epsilon_{2}}\xi^{\epsilon_{1}}}{\left|\xi\right|^{3}}\right) + \frac{iC}{2}\left(\frac{\xi^{\epsilon_{1}}\xi^{\epsilon_{2}}}{\left|\xi\right|^{3}}\right)+\frac{C}{2}\left(\frac{1}{\left|\xi\right|}-\frac{\xi^{\epsilon_{2}}\xi^{\epsilon_{2}}}{\left|\xi\right|^{3}}\right)\right\}\\
&=\mbox{diag}\left\{\frac{C}{2}\frac{1}{\left|\xi\right|}\right\}.
\end{align*}
where $C>0$. To explain this a little more thoroughly consider the formula for $\left(\tilde{N}^{00}_{A,\Phi}f\right)(x)$ given above. Essentially we are integrating a vector valued function $f(\varphi_{x,v}(t))$ against a matrix 
$$B(x,v,t):=\tilde{U}_{A,\Phi}(x,v)\tilde{U}^{-1}_{A,\Phi}\left(\varphi_{x,v}(t)\right).$$
\cite[Lemma B.1]{DPSU} states that we can compute the principal symbol by integrating the matrix $B(x,v,t)\big |_{t=0}$ against a certain function specified in the lemma. Now 
\begin{align*}
B(x,v,t)\big |_{t=0} &= \tilde{U}_{A,\Phi}(x,v)\tilde{U}^{-1}_{A,\Phi}\left(\varphi_{x,v}(0)\right)\\
                     &= \tilde{U}_{A,\Phi}(x,v)\tilde{U}^{-1}_{A,\Phi}\left(x,v\right)\\
                     &= \operatorname{Id}
\end{align*}
Therefore, the presence of the attenuation cancels in the symbol computation, and the formula for the principal symbol given above now follows for exactly the same reasons as in the unattenuated case in \cite{DPSU}. Similar reasoning applies to the other operators $\left(\tilde{N}^{01}_{A,\Phi}\eta\right)$ etc. above. The explicit formulas for the principal symbols are written down in \cite[Proposition 7.2]{DPSU}.\\
\indent Note that the fact that the symbols are independent of the attenuation depends upon the attenuation being unitary. Otherwise the matrix $\left(\tilde{U}_{A,\Phi}^{-1}\right)^{*}$ would occur inside the integrals, and this would not cancel upon application of \cite[Lemma B.1]{DPSU}.

\section{Proof of Theorem~\ref{main}}\label{5}
Let $w^\sharp$ be any smooth solution of the transport equation $\mathbf G_\mu w^\sharp+(A+\Phi)w^\sharp=0$. Applying \eqref{[H,X]} to $w^\sharp$ we get
$$
-(\mathbf G_\mu+A+\Phi)\mathcal H w^\sharp=(X_\perp+\star A)w_0^\sharp+(X_\perp w^\sharp+\star Aw^\sharp)_0.
$$
Since $X_\perp f=\star df$ for $f\in\Omega_0$ we have
$$
(X_\perp+\star A)w_0^\sharp=\star d_A w^\sharp_0.
$$
Since $X_\perp=i(\eta_--\eta_+)$ and $\star(A_{-1}+A_1)=i(A_{-1}-A_1)$ we obtain
$$
(X_\perp w^\sharp+\star Aw^\sharp)_0=i(\eta_-w^\sharp_1-\eta_+w^\sharp_{-1})+i(A_{-1}w^\sharp_{1}-A_1w^\sharp_{-1})=i(\mu_-w^\sharp_1-\mu_+w^\sharp_{-1})
$$
where $\mu_\pm=\eta_\pm+A_{\pm1}$. According to \cite[Lemma~6.2]{PSU2} the following holds
$$
\star d_A \alpha=2i(\mu_-\alpha_1-\mu_+\alpha_{-1}).
$$
Collecting everything together and using \eqref{Remark1} \& \eqref{Remark2} we have
$$
-2\pi(\mathbf G_\mu+A+\Phi)\mathcal H w^\sharp=2\pi \star d_Aw^\sharp_0+\pi\star d_A(w^\sharp_{-1}+w^\sharp_1).
$$
Applying $I_{A,\Phi}$ to the above equality we obtain
\begin{equation}\label{final-equation}
-2\pi P_{A,\Phi}=I^{0}_{A,\Phi} \star d_{A} (I^{1}_{A,\Phi})^{*}+I^{1}_{A,\Phi} \star d_{A} (I^{0}_{A,\Phi})^{*}.
\end{equation}

We also need the following result whose proof is postponed until after the proof of Theorem~\ref{main}.
\begin{Lemma}\label{final2}
Let $(M,g)$ be a Riemannian disk, $A$ a unitary connection and $\Phi$ a skew-Hermitian Higgs field.
\begin{itemize}
\item[(a)] Let $\alpha$ be a smooth $\C^n$-valued 1-form. Then there exist functions $a,p\in C^\infty(M,\C^n)$ and $\eta\in\mathfrak H_A$ such that $p\big|_{\p M} = 0$ and $d_A p+\star d_A a+\eta=\alpha$.
\item[(b)] Given $f,a\in C^\infty(M,\C^n)$ there is a smooth $\C^n$-valued 1-form $\beta$ with $\star d_A\,\beta=f$ and $d_A^*\beta=\Phi a$.
\end{itemize}
\end{Lemma}

\begin{proof}[Proof of Theorem~\ref{main}]
Suppose that $u=P_{A,\Phi}w+I_{A,\Phi}^1\eta$ for $\eta\in\mathfrak H_A$, then \eqref{final-equation} shows that $u$ belongs to the range of $I^0_{A,\Phi}+I^1_{A,\Phi}$. Conversely, suppose $u=I^0_{A,\Phi}(f)+I^1_{A,\Phi}(w)$ for
some smooth $\C^n$-valued function $f$ and 1-form $w$ on $M$. By item (a) of Lemma~\ref{final2} we can find $a,p\in C^\infty(M,\C^n)$ with $p\big|_{\p M}=0$ and $\eta\in\mathfrak H_A$ such that
$$
d_A p+\star d_A a+\eta=w.
$$
Since $I_{A,\Phi}(d_A p+\Phi p)=0$, we have $u=I^1_{A,\Phi}(\eta)+I^1_{A,\Phi}(\star d_Aa)+I^{0}_{A,\Phi}(f-\Phi p)$. By item (b) of Lemma~\ref{final2} we can find $\C^n$-valued 1-form $\beta$ such that $\star d_A\beta=f-\Phi p$ and $d_A^*\beta=2\Phi a$. Therefore we obtain $u=I^1_{A,\Phi}(\eta)+I^1_{A,\Phi}(\star d_A a)+I^0_{A,\Phi}(\star d_A\beta)$. By Theorem~\ref{surjectivity of I_A^*} there is $w\in\mathcal S_{A,\Phi}^\infty(\p_+ SM,\C^n)$ such that $(I_{A,\Phi}^1)^*(w)=\beta$ and $(I_{A,\Phi}^0)^*(w)=a$. Using \eqref{final-equation} we conclude that $u=I^1_{A,\Phi}(\eta)-2\pi P_{A,\Phi}(w)$.
\end{proof}

Item (a) of Lemma~\ref{final2} was proved in \cite[Lemma~6.1-(1)]{PSU2}. To prove item (b) of Lemma~\ref{final2}, consider the Laplacian corresponding to $d_A$
$$
-\Delta_A=d_A^*d_A+d_A d_A^*.
$$
This operator acts on $\C^n$-valued graded forms and maps $k$-forms to $k$-forms. The following result directly implies item (b) of Lemma~\ref{final2}.
\begin{Lemma}
Given $f,a\in C^\infty(M,\C^n)$ there is a smooth $\C^n$-valued 1-form $\beta$ with $d_A^*\beta=f$ and $d_A\beta=\Phi a\,\, d\Vol_g$.
\end{Lemma}
\begin{proof}
The proof is essentially identical to the proof of \cite[Lemma~6.6]{PSU2}. Look for $\beta$ of the form $\beta=d_A u^0+d_A^* u^2$ where $u^0,u^2$ are smooth forms. Then we need $u^0,u^2$ to satisfy
$$
d_A^* d_A u^0+(F_A)^* u^2=f,\quad F_A u^0+d_A d_A^* u^2=\Phi a\,\,d\Vol_g,
$$
where $F_A=d_A\circ d_A=dA+A\wedge A$ is the curvature of $d_A$. Writing $u=u^0+u^2$, these equations are equivalent to
$$
(-\Delta_A+R)u=f+\Phi a\,\,d\Vol_g,
$$
for some operator $R$ of order $0$. Then \cite[Lemma~6.5]{PSU2} implies the existence of a smooth solution $u$, hence the existence of the desired $\beta$.
\end{proof}

\subsection*{Acknowledgements.}
The first author wishes to thank his advisor, Gabriel Paternain, for all his encouragement and support, and Mikko Salo for various helpful comments which improved this article. The second author would like to express his acknowledgements to Professor Gunther Uhlmann, for constant assistance and support. The work of the second author was partially supported by NSF.


\begin{thebibliography}{ABC}
\bibitem{Ain} G. Ainsworth, \emph{The attenuated magnetic ray transform on surfaces}, Inverse Probl. Imaging, \textbf{7} no. 1 (2013), 27--46.

\bibitem{AS} D. V. Anosov, Y. G. Sinai, \emph{Certain smooth ergodic systems} [Russian], Uspekhi Mat. Nauk, \textbf{22} (1967), 107--172.
     
\bibitem{Ar} V. I. Arnold, \emph{Some remarks on flows of line elements and frames}, Sov. Math. Dokl., \textbf{2} (1961), 562--564.

\bibitem{ArG} V. I. Arnold, A. B. Givental, ``Symplectic Geometry", Dynamical Systems IV, Encyclopaedia of Mathematical Sciences, Springer Verlag, Berlin, 1990.

\bibitem{Bo} N. Bourbaki, \emph{Topological vector spaces},	Springer-Verlag, Berlin, 1987.

\bibitem{DPSU} N. S. Dairbekov, G. P. Paternain, P. Stefanov, G. Uhlmann, \emph{The boundary rigidity problem in the presence of a magnetic field}, Adv. Math., \textbf{216} (2007), 535--609.

\bibitem{DU} N. Dairbekov, G. Uhlmann, \emph{Reconstructing the Metric and Magnetic Field from the Scattering Relation}, Inverse Probl. Imaging, \textbf{4} (2010), 397--409.

\bibitem{Dun} M. Dunajski, {\it Solitons, instantons, and twistors}. Oxford Graduate Texts in Mathematics, 19. Oxford University Press, Oxford, 2010.

\bibitem{HSW} N. J. Hitchin, G. B. Segal, R. S. Ward, \emph{Integrable systems. Twistors, loop groups, and Riemann surfaces}, Oxford Graduate Texts in Mathematics, 4. The Clarendon Press, Oxford University Press, New York, 1999.

\bibitem{GK} V. Guillemin, D. Kazhdan, \emph{Some inverse spectral results for negatively curved 2-manifolds}, Topology \textbf{19} (1980), 301--312.

\bibitem{Juh} P. Juhlin, 1992, Principles of Doppler tomography, Tekniska Hogskolan i Lund, Matematiska Institutionen, LUTFD2/TFMA-92/7002.

\bibitem{Kb} S. Kobayashi, \emph{Differential geometry of complex vector bundles}, Publications of the Mathematical Society of Japan 15, Kan\^o Memorial Lectures 5, Princeton University Press, Princeton, NJ; Iwanami Shoten, Tokyo, 1987. 

\bibitem{Ko} V. V. Kozlov, \emph{Calculus of variations in the large and classical mechanics}, Russian Math. Surveys, \textbf{40} (2011), no. 2, 37--71.

\bibitem{MS} N. Manton, P. Sutcliffe, Topological solitons. Cambridge Monographs on Mathematical Physics, Cambridge University Press, Cambridge, 2004.

\bibitem{MW} L.J. Mason, N. M. J. Woodhouse, \emph{Integrability, self-duality, and twistor theory}, London Mathematical Society Monographs. New Series, 15. Oxford Science Publications. The Clarendon Press, Oxford University Press, New York, 1996.

\bibitem{Mi} R. Michel, \emph{Sur la rigidit\'e impos\'ee par la longueur des g\'eod\'esiques}, Invent. Math., \textbf{65} (1981), 71--83.

\bibitem{N1} S. P. Novikov, \emph{Variational methods and periodic solutions of equations of Kirchhoff type. II}, Funct. Anal. Appl., \textbf{15} (1981), 263--274.

\bibitem{N2} S. P. Novikov, \emph{Hamiltonian formalism and a multivalued analogue of Morse theory}, Russian Math. Surveys, \textbf{37} (1982), no. 5, 1--56.

\bibitem{NS} S. P. Novikov, I. Shmel'tser, \emph{Periodic solutions of the Kirchhoff equations for the free motion of a rigid body in a liquid, and the extended Lyusternik-Schnirelmann-Morse theory. I. }, J. Functional Anal. Appl., \textbf{15} (1981), 197--207.

\bibitem{Pa} G. P. Paternain, {\it Transparent connections over negatively curved surfaces}, J. Mod. Dyn. \textbf{3} (2009), 311--333.

\bibitem{PP} G. P. Paternain, M. Paternain, \emph{Anosov geodesic flows and  twisted symplectic structures}, in International Congress on Dynamical Systems in Montevideo (a tribute to Ricardo Ma\~n\'e), F. Ledrappier, J. Lewowicz, S. Newhouse eds, Pitman Research Notes in Math. \textbf{362} (1996), 132--145.

\bibitem{PSU3} G. P. Paternain, M. Salo, G. Uhlmann, \emph{Spectral rigidity and invariant distributions on Anosov surfaces}, preprint, arXiv: 1208.4943, to appear in J. Diff. Geom.

\bibitem{PSU2} G. P. Paternain, M. Salo, G. Uhlmann, \emph{On the range of the attenuated ray transform for unitary connections}, Int. Math. Res. Not., (online), (2013).

\bibitem{PSU1} G. P. Paternain, M. Salo, G. Uhlmann, \emph{Tensor tomography on surfaces}, Invent. Math., \textbf{193} (2013), 229--247. 


\bibitem{PSU} G. P. Paternain, M. Salo, G. Uhlmann, \emph{The attenuated ray transform for connections and Higgs fields}, Geom. Funct. Anal. \textbf{22} (2012), 1460--1489.

\bibitem{PU1} L. Pestov, G. Uhlmann, \emph{On the characterization of the range and inversion formulas for the geodesic X-ray transform}, Int. Math. Res. Not., \textbf{80} (2004), 4331--4347.

\bibitem{PU} L. Pestov, G. Uhlmann, \emph{Two dimensional compact simple Riemannian manifolds are boundary distance rigid}, {\em Ann. of Math.} \textbf{161} no. 2 (2005), 1089--1106.

\bibitem{Po} E. Powell, \emph{Boundary Rigidity}, unpublished draft, 2014.

\bibitem{SaU} M. Salo, G. Uhlmann, \emph{The attenuated ray transform on simple surfaces}, J. Diff. Geom. \textbf{88} (2011), no. 1, 161--187.

\bibitem{St} P. Stefanov, \emph{Personal communication}, 12/02/2014.

\bibitem{T} M. E. Taylor, Partial Differential Equations I. Basic Theory. Second edition. Applied Mathematical Sciences, 115. Springer, New York, 2011.

\bibitem{Tr} F. Treves, \emph{Topological vector spaces, distributions and kernels}, Academic Press, New York, 1967.
  
\end{thebibliography}
\end{document}